\documentclass[12pt,a4paper,oneside,draft,hrule]{article}
\usepackage[hmargin=2.5cm,vmargin=3cm]{geometry}
\usepackage{amssymb}
\usepackage{amsmath}
\usepackage{amsthm}
\usepackage{amsfonts}
\usepackage{graphicx}
\usepackage{euscript}
\usepackage{amscd}
\usepackage{exscale}
\usepackage{fancyhdr}
\usepackage{setspace}
\numberwithin{equation}{section}
\newcommand{\nt}{\noindent}
\newcommand{\ds}{\displaystyle}

\newcommand{\N}{\mathbb{N}}

\newtheorem{theorem}{Theorem}[section]
\newtheorem{lemma}[theorem]{Lemma}
\newtheorem{cor}[theorem]{Corollary}
\theoremstyle{definition}
\newtheorem{remark}[theorem]{Remark}

\title{Champagne subregions of the unit disc}
\author{Joanna Pres\footnote{School of Mathematical Sciences, University College Dublin, Belfield, Dublin 4, Ireland; joanna.t.pres@gmail.com, joanna.pres@ucd.ie}}
\date{}
\begin{document}
\maketitle
\begin{abstract}
This paper concerns harmonic measure on the domains  that arise when infinitely many disjoint closed discs are removed from the unit disc. It investigates which configurations of discs  are unavoidable for Brownian motion, and obtains  refinements of related results of Akeroyd, and of Ortega-Cerd\`{a} and Seip.
\smallskip


\end{abstract}

\section{Introduction}

Let $B(x,r)$ denote the open ball of centre $x$ and radius $r>0$ in $\mathbb{R}^n$ ($n\geq 2$), and let $B=B(0,1)$. We consider a  sequence of pairwise disjoint closed balls $\overline{B}(x_k,r_k)$ such that $|x_k|\to 1$ and $\sup r_k/(1-|x_k|)<1$, and use them to form a \emph{champagne subregion} $\Omega=B\setminus E$ of the unit ball, where $E=\bigcup_{k=1}^\infty \overline{B}(x_k,r_k)$. We assume for convenience that $0\in\Omega$. We say that $E$ is \emph{unavoidable} if it carries full harmonic measure with respect to  $\Omega$, that is, the harmonic measure of $\partial B$ relative to $\Omega$ is zero. This can be viewed as the case where the probability that Brownian motion starting from the origin first exits $\Omega$ through $\partial B$  is zero.

Akeroyd \cite{ak}  has shown that the following phenomenon occurs.
\medskip

\nt \textbf{Theorem A}
\emph{Let $\varepsilon>0$. There is a  champagne subregion of the unit disc $B$ such that $\sum_k r_k<\varepsilon$ and yet $E$ is unavoidable.}
\medskip

\nt Ortega-Cerd\`{a} and Seip \cite{ortse} provided a description of unavoidable configurations of discs in $B$ and improved Akeroyd's result.
\medskip

\nt \textbf{Theorem B}
\begin{enumerate}
\item[(a)] \emph{Let $(x_k)$  be a sequence  in $B\subset \mathbb{R}^2$ satisfying}
\begin{equation}\label{sep0}
\inf_{j\neq k}\cfrac{|x_j-x_k|}{1-|x_k|}>0
\end{equation}
\emph{and}
\begin{equation}\label{dens}
B(x,a(1-|x|))\cap \{x_k: k\in\mathbb{N}\}\neq \emptyset \quad (x\in B)
\end{equation}
\emph{for some $a\in (0,1)$. Let $r_k=(1-|x_k|)\phi(|x_k|)$ for some decreasing function $\phi: [0,1)\to (0,1)$. Then $E$ is unavoidable if and only if }
\[
\int_0^1 \cfrac{1}{(1-t)\log(1/\phi(t))}dt=\infty.
\]
\item[(b)] \emph{For any $(x_k)$ satisfying \eqref{sep0} and \eqref{dens}, and for any $\varepsilon>0$, there exists a  champagne subregion of $B\subset \mathbb{R}^2$ such that $\sum_k r_k<\varepsilon$ and $E$ is unavoidable.}
\end{enumerate}
Recently Gardiner and Ghergu \cite{gargh} obtained the following result for more general champagne subregions when $n\geq 3$, where the separation condition \eqref{sep0} is relaxed. Let $\sigma$ denote normalised surface area measure on $\partial B$ (normalised arclength on $\partial B$, if $n=2$).
\bigskip

\nt \textbf{Theorem C} \emph{Let $\Omega$ be a champagne subregion of the unit ball in $\mathbb{R}^n$ ($n\geq 3$).}
\begin{enumerate}
\item[(a)] \emph{If $E$ is unavoidable, then}
\begin{equation}\label{gar}
\sum_k \frac{(1-|x_k|)^2}{|y-x_k|^2}r_k^{n-2}=\infty \quad for \ \sigma-almost \ every \ y\in\partial B.
\end{equation}
\item[(b)] \emph{If \eqref{gar} holds together with the separation condition }
\begin{equation}\label{gar2}
\inf_{j\neq k}\cfrac{|x_j-x_k|}{r_k^{1-2/n}(1-|x_k|)}>0,
\end{equation}
\emph{then $E$ is unavoidable.}
\end{enumerate}

If we substitute $n=2$ into \eqref{gar2} we obtain \eqref{sep0}, so it would be natural to assume that the latter is the appropriate separation condition to use in the case of the disc. However, a more careful analysis of the plane case yields a stronger, and less obvious, result. This will, in turn, lead to interesting refinements of Theorems A and B.

\begin{theorem}\label{main}Let $\Omega=B\setminus E$ be a champagne subregion of the unit disc $B$.
\begin{itemize}
\item[(a)]   If $E$ is unavoidable, then
\begin{equation}\label{nec}
\sum_k \frac{(1-|x_k|)^2}{|y-x_k|^2} \left\{\log\frac{1-|x_k|}{r_k}\right\}^{-1}=\infty \mbox{ for } \sigma-almost \ every \ y\in\partial B.
\end{equation}
In particular,
\begin{equation*}
\sum_k \frac{(1-|x_k|)^2}{|y-x_k|^2}=\infty \quad for \ \sigma-almost \ every \ y\in\partial B.
\end{equation*}
\item[(b)] Conversely, if \eqref{nec} holds together with the separation condition
\begin{equation}\label{sep}
\inf_{j\neq k}\frac{|x_k-x_j|\left\{\log\frac{1-|x_k|}{r_k}\right\}^{1/2}}{1-|x_k|}>0,
\end{equation}
 then $E$ is unavoidable.

\end{itemize}
\end{theorem}
 Using Theorem \ref{main} we  can describe the unavoidable configurations of discs for which  $r_k=(1-|x_k|)\phi(|x_k|)$, where $\phi: [0,1)\to (0,1)$ is decreasing.  Let $M:[0,1)\to[1,\infty)$ be an increasing function satisfying
\[
M(1-t/2)\leq cM(1-t) \quad (0<t\leq 1)
\]
for some $c>1$. The number of centres in a given disc will be denoted by
\[
N_a(x)=\# [B(x,a(1-|x|))\cap \{x_k: k\in \mathbb{N}\}] \quad \mbox{ for some } a\in (0,1).
\]
We can now present our refinement of Theorem B(a).
 \begin{theorem}\label{neces2} Let $\Omega=B\setminus E$ be a champagne subregion of the unit disc, where $r_k=(1-|x_k|)\phi(|x_k|)$.
 \begin{itemize}\item[(a)]If $E$ is unavoidable and there are constants $a\in (0,1)$ and $b>0$ such that $N_a(x)\leq bM(|x|)$ for all $x\in B$,
 then
 \begin{equation}\label{integral}
\int_0^1 \cfrac{M(t)}{(1-t)\log(1/\phi(t))}dt =\infty.
\end{equation}
\item[(b)] Conversely, if \eqref{integral} holds together with
\begin{equation}\label{sep3}
\inf_{j\neq k}\frac{|x_k-x_j|\left\{\log\frac{1}{\phi(|x_k|)}\right\}^{1/2}}{1-|x_k|}>0,
\end{equation} and if there are constants $a\in (0,1)$ and $b>0$ such that $
N_a(x)\geq bM(|x|)$ for all $x\in B$,
then $E$ is unavoidable.
\end{itemize}
\end{theorem}
From this we deduce a further strengthening of Theorem A.
\begin{cor}\label{ake}Let $\varepsilon >0$.
\begin{itemize}
\item[(a)]
 For any $\alpha>0$ there is a champagne subregion of the unit disc such that $\sum_k r_k^{\alpha}<\varepsilon$ and $E$ is unavoidable.
 \item[(b)] For any $\alpha>1$ there is a champagne subregion of the unit disc such that
 \[\sum_k \left\{\log\frac{1}{r_k}\right\}^{-\alpha} <\varepsilon\] and $E$ is unavoidable.
\end{itemize}
\end{cor}
We observe that  for every $\alpha>0$ there exists a constant $c(\alpha)>0$ such that $r^\alpha\leq c(\alpha) \{\log(1/r)\}^{-2}$ for all $r\in(0,1)$. Thus  (b) implies (a) in Corollary \ref{ake}.

\begin{remark}\label{a=1} The conclusion of Corollary \ref{ake}(b) fails when $\alpha=1$. For details see the end of Section 3.
\end{remark}
We will establish the above results by combining the methods of Gardiner and Ghergu \cite{gargh} with some new ideas. From now on we will assume that the dimension $n$ is $2$. We refer to \cite{armgar} for the relevant background material on potential theory.

 \section{Proof of Theorem \ref{main}}
For a positive superharmonic function $u$ on $B$ we define the reduced function of $u$ relative to a subset $E$ of $B$
\[
R_u^E=\inf\{v: v \mbox{ is positive and superharmonic on } B \mbox{ and } v\geq u \mbox{ on } E\}.
\]
The Poisson kernel for $B$ is given by
\[P(x,y)=\frac{1-|x|^2}{|x-y|^2} \ \mbox{ for } y\in \partial B \mbox{ and }x\in B.
 \]
 In view of \cite[Theorem 6.9.1]{armgar}, $E$ is unavoidable if and only if $R_1^E(0)=1$. By the use of the same theorem and the fact that $\int_{\partial B} P(\cdot,y)d\sigma(y)=1$ it can be seen that $E$ is unavoidable if and only if $R^E_{P(\cdot,y)}(0)=1=P(0,y)$ for $\sigma$-almost every $y\in \partial B$. Since $\Omega$ is connected and  $P(\cdot,y)-R^E_{P(\cdot,y)}$ ($y\in\partial B$) is a nonnegative harmonic function on $\Omega$, it follows from the maximum principle that
\begin{equation}\label{unav}
E \mbox{ is unavoidable if and only if } R^E_{P(\cdot,y)}\equiv P(\cdot,y) \mbox{ for } \sigma-\mbox{almost every } y\in\partial B.
\end{equation}
Now for $n\in \N$ and $m\in \mathbb{Z}$ such that $0\leq m<2^{n+4}$ let
\[
S_{m,n}=\left\{re^{i\theta}: 2^{-n-1}\leq 1-r\leq 2^{-n} \mbox{ and } \frac{2\pi m}{2^{n+4}}\leq \theta\leq \frac{2\pi(m+1)}{2^{n+4}}\right\}
\]
and
\[
z_{m,n}=(1-2^{-n})\exp(2\pi i m/2^{n+4}).
\]
It is easy to see that the diameter ${\rm diam}(S_{m,n})$ of $S_{m,n}$ satisfies
\begin{equation}\label{diam}
\ds \cfrac{2^{-n}}{2}\leq{\rm diam}(S_{m,n})\leq \cfrac{2^{-n}}{\sqrt{2}}.
\end{equation}
 Theorem 1 of \cite{essen} (cf. \cite[Corollary 7.4.4]{aies}) tells us that
\begin{equation}\label{log}
R^E_{P(\cdot,y)}\equiv P(\cdot,y) \mbox{ if and only if } \sum_{m,n}\left\{\frac{2^{-n}}{|z_{m,n}-y|}\right\}^2\left\{\log \frac{2^{-n}}{c(E\cap S_{m,n})}\right\}^{-1}=\infty,
\end{equation}
where $c(\cdot)$ denotes logarithmic capacity, and $\left\{\log \frac{2^{-n}}{c(E\cap S_{m,n})}\right\}^{-1}$ is interpreted as 0 whenever $E\cap S_{m,n}$ is polar.

 Before commencing the proof of Theorem \ref{main} we  state the following lemma.

\begin{lemma} \label{sup}There is a constant $c_1>1$ such that, for any $S_{m,n}$ and any $\overline{B}(x_k,r_k)$ which intersects $S_{m,n}:$
\[
c_1^{-1}2^{-n}\leq 1-|x_k|\leq c_1 2^{-n}, \ \mbox{ and } \ c_1^{-1}\leq \frac{|z_{m,n}-y|}{|x_k-y|}\leq c_1 \mbox{ for all } y\in\partial B.
\]
The constant $c_1$ depends on the value of $\sup \{r_k/(1-|x_k|)\}$.
\end{lemma}
The proof of Lemma \ref{sup} is straightforward  and hence omitted.

We shall also  make use of the elementary observation that, if $p>1$ then $f(t)=\log(pt)/\log t$ $(t>1)$ is decreasing.
\begin{proof}[Proof  of Theorem  \ref{main}(a)] Suppose that $E$ is unavoidable. Applying \cite[Theorem 5.1.4 (a)]{ran}, concerning a relation between the logarithmic capacity and unions, and Lemma \ref{sup} we obtain
\begin{equation*}
\begin{aligned}
\sum_{m,n}\left\{\frac{2^{-n}}{|z_{m,n}-y|}\right\}^2&\left\{\log \frac{2^{-n}}{c(E\cap S_{m,n})}\right\}^{-1}\\
&
\leq \sum_k \sum_{m,n}\left\{\frac{2^{-n}}{|z_{m,n}-y|}\right\}^2\left\{\log \frac{2^{-n}}{c(\overline{B}(x_k,r_k)\cap S_{m,n})}\right\}^{-1}\\
&\leq c_1^4\sum_k \frac{(1-|x_k|)^2}{|x_k-y|^2}\sum_{m,n}\left\{\log \frac{2^{-n}}{c(\overline{B}(x_k,r_k)\cap S_{m,n})}\right\}^{-1}\\
& \leq C(c_1)\sum_k \frac{(1-|x_k|)^2}{|x_k-y|^2}\sum_{m,n}\left\{\log \frac{c_12^{-n}}{c(\overline{B}(x_k,r_k)\cap S_{m,n})}\right\}^{-1}.
\end{aligned}
\end{equation*}
The last inequality can be deduced from the observation following Lemma \ref{sup} and \eqref{diam} combined with the fact that $c(\overline{B}(x_k,r_k)\cap S_{m,n})\leq {\rm diam}(S_{m,n})$.
Since a given disc $\overline{B}(x_k,r_k)$ intersects at most $c_2$ of the sets $S_{m,n}$ (where $c_2$ is independent of $k$ and of $m,n$) and by  the fact that $c(\overline{B}(x_k,r_k)\cap S_{m,n})\leq c(\overline{B}(x_k,r_k))=r_k$, we obtain
\begin{equation*}
\begin{aligned}
\sum_{m,n}\left\{\frac{2^{-n}}{|z_{m,n}-y|}\right\}^2\left\{\log \frac{2^{-n}}{c(E\cap S_{m,n})}\right\}^{-1}
&\leq C(c_1)c_2\sum_k \frac{(1-|x_k|)^2}{|y-x_k|^2}\left\{\log\frac{1-|x_k|}{r_k}\right\}^{-1}.
\end{aligned}
\end{equation*}
 In view of \eqref{unav} and \eqref{log} we see that \eqref{nec} holds if $E$ is unavoidable.
\end{proof}

\subsection{A quasiadditivity property}
Following \cite{aibor}, we define the capacity $\mathcal{C}_2$ by
\[
\mathcal{C}_2(A)=\inf\left\{\mu(\mathbb{R}^2):  \int \log^+\frac{2}{|x-y|}d\mu(y)\geq 1 \mbox{ on } A\right\}.
\]
Proposition 4.1.1 of \cite{aies} tells us that $\mathcal{C}_2(B(x,r))$ is comparable with
\[
\left\{ \frac{1}{r^2}\int_0^r t\log^+\frac{2}{t}dt\right\}^{-1},
\]
from which it follows that there exists a constant $c_3>1$ such that
\begin{equation}\label{estball}
c_3^{-1} \left\{\log \cfrac{1}{r}\right\}^{-1}\leq\mathcal{C}_2(B(x,r))\leq c_3 \left\{\log \cfrac{1}{r}\right\}^{-1} \quad (0<r\leq 1/2).
\end{equation}
Let $|A|$ denote  the Lebesgue measure of a planar set $A$, and for $\alpha>0$ let $\alpha A=\{\alpha x: x\in A\}$. We write $A^\circ$ to denote the interior of $A$. We recall  the following quasiadditivity property of the capacity $\mathcal{C}_2$ (see \cite[Theorem 3]{aibor}).
\medskip

\nt \textbf{Theorem D}
\emph{For $k\in \mathbb{N}$ and $\rho_k>0$ let $\eta(\rho_k)$ be such that $|B(y_k,\eta(\rho_k))|=\mathcal{C}_2(B(y_k,\rho_k))$. Let $\eta^*(\rho_k)=\max\{\eta(\rho_k), 2\rho_k\}$. If the discs $\{B(y_k, \eta^*(\rho_k))\}$ are pairwise disjoint and $F$ is an analytic subset of $\bigcup_k B(y_k,\rho_k)$, then there is a constant $c_4>0$ such  that}
\begin{equation}\label{F}
\mathcal{C}_2(F)\geq c_4\sum_k \mathcal{C}_2(F\cap B(y_k,\rho_k)).
\end{equation}

\nt Using \eqref{estball} we deduce that if $F$ is an analytic subset of $\bigcup_k B(y_k,\rho_k)$, where $\rho_k\leq 1/(4\pi c_3)$ for all $k$, and if the discs
\[
\left\{B\left(y_k, \left[\pi c_3 \log (1/\rho_k)\right]^{-1/2}\right): k\in \mathbb{N}\right\}
\]
are pairwise disjoint, then \eqref{F} holds.

\begin{lemma}\label{alpham} Suppose that
\begin{equation}\label{quasisep}
  \frac{|x_j-x_k|\left(\log \frac{1-|x_k|}{r_k}\right)^{1/2}}{1-|x_k|}\geq 8\sqrt{\pi}c_3 c_1^2 \quad (j\neq k),
\end{equation}
where $c_1$ is as in Lemma \ref{sup} and $c_3$ as in \eqref{estball}. Then,  for each ,
\begin{equation}\label{quasiad}
\mathcal{C}_2(\alpha_n[\cup_k B(x_k,r_k)\cap S_{m,n}^\circ])\geq c_4 \sum_k \mathcal{C}_2(\alpha_n[B(x_k,r_k)\cap S_{m,n}^\circ]),
\end{equation}
where $\alpha_n=(4\pi c_1 c_3)^{-1}2^n$  and $c_4$ is as in Theorem D.
\end{lemma}

\begin{proof} Let $F=\cup_k [B(y_k,\rho_k)\cap \alpha_n S_{m,n}^\circ]$, where $y_k=\alpha_n x_k$ and $\rho_k=\alpha_n r_k$. We observe that
\[
\rho_k=\alpha_n r_k\leq \alpha_n (1-|x_k|)\leq \alpha_n c_1 2^{-n}=\cfrac{1}{4\pi c_3}.
\]
For $k$ such that $\overline{B}(x_k, r_k)\cap S_{m,n}\neq \emptyset$ we also have
\[
\log \cfrac{1}{\rho_k}=\log \cfrac{4\pi c_3 c_1 2^{-n}}{r_k} \geq \log\cfrac{1-|x_k|}{r_k}.
\]
For $j\neq k$, we thus see from \eqref{quasisep} that
\[
\begin{aligned}
\cfrac{|y_k-y_j|}{2 \{\pi c_3 \log (1/\rho_k)\}^{-1/2}}
&=\cfrac{\alpha_n |x_k-x_j|\{\log (1/\rho_k)\}^{1/2}}{2(\pi c_3)^{-1/2}}\\
&\geq \cfrac{\alpha_n|x_k-x_j|\left\{\log\frac{1-|x_k|}{r_k}\right\}^{1/2}}{2(\pi c_3)^{-1/2}}\\
&\geq  8\sqrt{\pi}c_3 c_1^2 (1-|x_k|)\cfrac{\alpha_n}{2(\pi c_3)^{-1/2}}\\
&\geq 4\pi c_1 c_3^{3/2} 2^{-n} \alpha_n\\
&=c_3^{1/2}> 1.
\end{aligned}
\]
Hence the required inequality \eqref{quasiad} holds by \eqref{F}.
\end{proof}

\begin{remark}\label{rem}
 (a)
 Suppose that $E\cap S_{m,n}$ is non-polar.
 We  know by \cite[Lemma 5.8.1]{armgar} that there exists a unit measure $\nu$ on  $\partial (\alpha_n [E\cap S_{m,n}])$ such that
 \[
 -\log c(\alpha_n [E\cap S_{m,n}])=\int\log\frac{1}{|x-y|}d\nu(y)
 \]
 for each $x\in \alpha_n(E^{\circ}\cap S_{m,n}^\circ)$. Since, for $x,y \in \alpha_n (E\cap S_{m,n})$, we have
 \[
 |x-y|\leq \alpha_n {\rm diam}(S_{m,n})\leq \frac{1}{4\pi c_1c_3}<1,
 \]
 it follows that
 \[
 \int\log^+\frac{2}{|x-y|}d\nu(y)\geq -\log c(\alpha_n [E\cap S_{m,n}])>0 \ \mbox{ for }x\in \alpha_n(E^{\circ}\cap S_{m,n}^\circ).
 \]
 Hence
 \begin{equation}\label{c2}
 \begin{aligned}
 \mathcal{C}_2(\alpha_n [E^\circ\cap S_{m,n}^\circ])\leq \frac{\nu(\mathbb{R}^2)}{-\log c(\alpha_n [E\cap S_{m,n}])}
 = \frac{1}{\log\frac{1}{c(\alpha_n [E\cap S_{m,n}])}}.
 \end{aligned}
 \end{equation}
 Furthermore, since $c(\alpha_n [E\cap S_{m,n}])=\alpha_n c(E\cap S_{m,n})$, we have
 \begin{equation}\label{login}
  \mathcal{C}_2(\alpha_n [E^\circ\cap S_{m,n}^\circ])\leq \left\{\log \cfrac{2^{-n}}{c(E\cap S_{m,n})}\right\}^{-1}.
 \end{equation}
\medskip

 (b) Let $k\in \mathbb{N}$.
Suppose that $r_k\leq (2^5c_1)^{-1} (1-|x_k|)$, that is, the diameter of $B(x_k,r_k)$ is not greater than the length of the shortest line segment joining points of $S_{m,n}$, where $S_{m,n}\cap B(x_k,r_k)\neq \emptyset$. Then $B(x_k, r_k)$ intersects at most four sets $S_{m,n}$, and one of them contains a disc of radius $r_k/4$. Recall that $\alpha_{n}=\{4\pi c_1 c_3\}^{-1}2^n$. Since $0<\alpha_n r_k<1/2$,
we can use  \eqref{estball} to obtain
  \[
 \begin{aligned}
 \mathcal{C}_2(\alpha_n[B(x_k,r_k)\cap S_{m,n}^\circ])&
 \geq c_3^{-1} \left\{\log \frac{4}{\alpha_n r_k}\right\}^{-1}\\
&= c_3^{-1} \left\{\log \frac{16\pi c_1 c_3 2^{-n}}{r_k}\right\}^{-1}\\
&\geq c_3^{-1} \left\{\log \frac{16\pi c_1^2 c_3 (1-|x_k|)}{r_k}\right\}^{-1}\\
&\geq C(c_1,c_3) \left\{\log \frac{1-|x_k|}{r_k}\right\}^{-1}.
 \end{aligned}
 \]
 In general, if $r_k\leq 1-|x_k|$, then $(2^5c_1)^{-1}r_k\leq (2^5c_1)^{-1} (1-|x_k|)$ and there exists $S_{m,n}$ that intersects $B(x_k,r_k)$ and contains a disc of radius $r_k/(2^7 c_1)$. Thus
 \[
 \begin{aligned}
 \mathcal{C}_2(\alpha_n[B(x_k,r_k)\cap S_{m,n}^\circ])&\geq c_3^{-1}\left\{\log \frac{128 c_1}{\alpha_n r_k}\right\}^{-1}\\
 &\geq C(c_1,c_3)\left\{\log \frac{4}{\alpha_n r_k}\right\}^{-1}\\
 &\geq C(c_1,c_3) \left\{\log \frac{1-|x_k|}{r_k}\right\}^{-1}.
 \end{aligned}
 \]
\end{remark}

\begin{proof}[Proof  of Theorem  \ref{main}(b)] Suppose that \eqref{nec} holds, together with separation condition \eqref{sep}. In view of \eqref{unav} and \eqref{log}, in order to prove that $E$ is unavoidable it is enough to show that, for $\sigma$-almost every $y\in\partial B$,
\[
\sum_{m,n}\left\{\frac{2^{-n}}{|z_{m,n}-y|}\right\}^2\left\{\log \frac{2^{-n}}{c(E\cap S_{m,n})}\right\}^{-1}=\infty.
\]
Let $\delta\in (0,1)$ be small enough so that for all $k$,
\[
\frac{|x_k-x_j|\left\{\log\frac{1-|x_k|}{\delta r_k}\right\}^{1/2}}{1-|x_k|}\geq 8\sqrt{\pi}c_3 c_1^2 \quad (j\neq k).\]
Let $E_\delta=\cup_k B(x_k,\delta r_k)$. We see from Lemma \ref{alpham}, Lemma \ref{sup} and Remark \ref{rem}(b) that
 \[
\begin{aligned}
\sum_{m,n}\left\{\frac{2^{-n}}{|z_{m,n}-y|}\right\}^2 &\mathcal{C}_2(\alpha_n[E_\delta\cap S_{m,n}^\circ])\\
&\geq c_4 \sum_k \sum_{m,n}\left\{\frac{2^{-n}}{|z_{m,n}-y|}\right\}^2  \mathcal{C}_2(\alpha_n[B(x_k,\delta r_k)\cap S_{m,n}^\circ])\\
&\geq C(c_1,c_3,c_4)\sum_k \frac{(1-|x_k|)^2}{|y-x_k|^2} \left\{\log \frac{1-|x_k|}{\delta r_k}\right\}^{-1}.
\end{aligned}
\]
Moreover, by the observation following Lemma \ref{sup}, we have $\log\frac{1-|x_k|}{\delta r_k}\leq c_5\delta^{-1}\log\frac{1-|x_k|}{r_k}$, where $c_5$ depends only on the value of $\sup\{r_k/(1-|x_k|)\}$. Hence
\[
\begin{aligned}
\sum_{m,n}\left\{\frac{2^{-n}}{|z_{m,n}-y|}\right\}^2 &\mathcal{C}_2(\alpha_n[E_\delta\cap S_{m,n}^\circ])\\
&\geq C(c_1,c_3,c_4, c_5)\delta\sum_k \frac{(1-|x_k|)^2}{|y-x_k|^2} \left\{\log \frac{1-|x_k|}{r_k}\right\}^{-1}.
\end{aligned}
\]
Now, the fact that $E_\delta\subseteq E^\circ$ and \eqref{nec} yield
\[
\sum_{m,n}\left\{\frac{2^{-n}}{|z_{m,n}-y|}\right\}^2\mathcal{C}_2(\alpha_n [E^\circ\cap S_{m,n}^\circ])=\infty \quad \mbox{ for } \sigma-\mbox{almost every } y\in\partial B.
\]
Therefore, by \eqref{login},
\[
\sum_{m,n}\left\{\frac{2^{-n}}{|z_{m,n}-y|}\right\}^2\left\{\log \frac{2^{-n}}{c(E\cap S_{m,n})}\right\}^{-1}=\infty,
\]
and we conclude that $E$ is unavoidable.
\end{proof}

\section{Proof of Theorem \ref{neces2} and Corollary \ref{ake}}
\begin{proof}[Proof of Theorem \ref{neces2}(a)] Suppose that $E$ is unavoidable and that $N_a(x)\leq b M(|x|)$ ($x\in B$). It follows from Theorem \ref{main}(a) that
\begin{equation}\label{nec2}
\sum_k \frac{(1-|x_k|)^2}{|y-x_k|^2} \left\{\log\frac{1}{\phi(|x_k|)}\right\}^{-1}=\infty \mbox{ for } \sigma-\mbox{almost every } y\in\partial B.
\end{equation}
Reasoning as in the proof of \cite[Theorem 2]{gargh} with $\{\phi(t)\}^{n-2}$ replaced by $\{\log (1/\phi(t))\}^{-1}$ we see that
\[
\int_{B}
\cfrac{ M((3|x|-2)^+)}{(1-|x|^2) \log \{1/\phi ((3|x|-2)^+)\}}dx=\infty.
\]
It  follows that
\[
\infty=\int_{2/3}^1 \cfrac{ M(3t-2)t}{(1-t^2) \log \{1/\phi (3t-2)\}}dt\leq \int_{2/3}^1 \cfrac{M(3t-2)}{(1-t) \log \{1/\phi (3t-2)\}}dt,
\]
and so \eqref{integral} holds.
\end{proof}

\begin{proof}[Proof of Theorem \ref{neces2}(b)] Again, by adapting the method from the proof of \cite[Theorem 2(b)]{gargh}, it can be shown that if $\Omega$ is a champagne subregion  with $r_k=(1-|x_k|)\phi(|x_k|)$ for which \eqref{integral} holds, and if there are constants $a\in (0,1)$ and $b>0$ such that  $N_a(x)\geq bM(|x|)$ for all $x\in B$,
then
\[
\sum_k \frac{(1-|x_k|)^2}{|y-x_k|^2} \left\{\log\frac{1-|x_k|}{r_k}\right\}^{-1}=\infty \quad (y\in\partial B).
\]
Since \eqref{sep3} corresponds to \eqref{sep}, it follows from Theorem \ref{main}(b) that $E$ is unavoidable.
\end{proof}

\begin{proof}[Proof of Corollary \ref{ake}] Since (b) implies (a) in Corollary \ref{ake} we only prove part (b).
Let $\varepsilon>0$ and $\alpha>1$. Let $\beta>1/(\alpha-1)$ and $\phi(t)=\exp\left(-\frac{1}{c_0(1-t)^\beta}\right)$ for $t\in [0,1)$ and some $c_0\in(0,1)$. Let $p_n$ be the integer part of $2^{n\beta/2}$.
We divide each \lq\lq square" $S_{m,n}$ into $p_n^2$ \lq\lq subsquares". 
Let $\{x_k:k\in\mathbb{N}\}$ be the collection of centres of those \lq\lq subsquares", and let $r_k=(1-|x_k|)\phi(|x_k|)$. Then the following observations can be made.
\begin{itemize}\item[(a)] Each $x_k$ belongs to some $S_{m,n}$ and so $1-|x_k|\leq 2^{-n}$. Moreover, since $|x_k-x_j|\geq \gamma 2^{-n}/p_n$ ($j\neq k$) for some universal constant $\gamma>0$, for $j\neq k$ we have
\[
\begin{aligned}
\frac{|x_k-x_j|\left\{\log\frac{1}{\phi(|x_k|)}\right\}^{1/2}}{1-|x_k|}
&=\frac{|x_k-x_j|}{\{c_0(1-|x_k|)^\beta\}^{1/2}(1-|x_k|)}\\
&\geq C(c_0,\gamma) \frac{2^{-n}}{p_n} \frac{1}{(1-|x_k|)^{\beta/2+1}}\\
 &\geq C(c_0,\gamma) \frac{2^{-n}}{2^{n\beta/2}}\frac{1}{(2^{-n})^{\beta/2+1}}\\
 &=C(c_0, \gamma).
\end{aligned}
\]

\item[(b)]  Let $M: [0,1)\to [1,\infty)$ be given by $M(t)=(1-t)^{-\beta}$. Then $M$ satisfies
\[
M(1-t/2)= 2^\beta M(1-t) \quad (t\in (0,1]),
\]
and clearly \eqref{integral} holds.

\item[(c)] If $0<a<1$ is sufficiently large, $N_a(x)\geq M(|x|)$ for all $x\in B$.
\item[(d)] If $c_0$ is sufficiently small, the closed discs $\overline{B}(x_k,r_k)$ are pairwise disjoint.
\end{itemize}
It follows from Theorem \ref{neces2}(b) that $E$ is unavoidable.

Finally, since $r_k\leq \exp\left(-\frac{1}{c_0(1-|x_k|)^\beta}\right)$, we deduce that
\[
\begin{aligned}
\sum_k\frac{1}{\left(\log \frac{1}{r_k}\right)^\alpha}
&\leq c_0^\alpha \sum_k (1-|x_k|)^{\alpha\beta}
\leq c_0^\alpha\sum_n (2^{-n})^{\alpha\beta}p_n^2 2^{n+4}
\leq 16c_0^\alpha\sum_n \left\{2^{-\beta(\alpha-1)+1}\right\}^n.
\end{aligned}
\]
Since $-\beta(\alpha-1)+1<0$, the above geometric series converges and so does $\sum_k\left\{\log(1/r_k)\right\}^{-\alpha}$. By omitting a finite number of the discs we can arrange that $\sum_k\left\{\log(1/r_k)\right\}^{-\alpha}<\varepsilon$ and yet the collection of discs is unavoidable.
\end{proof}

\nt\emph{Details of Remark \ref{a=1}}.
 To see that the conclusion of Corollary \ref{ake}(b) fails when $\alpha=1$, suppose that there is a champagne subregion $\Omega$ of the unit disc such that $\sum_k\{\log\frac{1}{r_k}\}^{-1}$ converges. Since $|x_k|\to 1$, by omitting a finite number of discs we can have $\overline{B}(x_k,r_k)\subset B\setminus B(0,1/2)$ for each $k$ and $\sum_k\{\log\frac{1}{r_k}\}^{-1}\leq 1/(2\log 4)$. Let $A$ be the union of all the remaining discs $\overline{B}(x_k,r_k)$. Since the Green capacity of $\overline{B}(x_k,r_k)$ relative to $B(0,2)$ is dominated by $1/\log(1/r_k)$ (see \cite[(5.8.5)]{armgar}), it follows that the value of the capacitary Green potential of $A$ on $B(0,2)$ is not bigger than $1/2$ at $0$. From this we deduce that $A$ is avoidable and so is the whole collection $\{\overline{B}(x_k,r_k): k\in \N\}$.

\bibliographystyle{abbrv}
\bibliography{database}
\end{document}